\documentclass[12,reqno]{amsart}

\usepackage{amsfonts}
\usepackage[active]{srcltx}
\usepackage{amsmath}
\usepackage{amssymb}
\usepackage{amsthm}
\usepackage{mathtools}
\usepackage{bm} 
\usepackage{empheq}
\usepackage{upgreek} 

\usepackage{enumitem}
\usepackage{wasysym}
\usepackage{verbatim}
\usepackage{graphicx}
\usepackage[
bookmarks=true,         
bookmarksnumbered=true, 
colorlinks=true, pdfstartview=FitV, linkcolor=blue, citecolor=blue,
urlcolor=blue]{hyperref}
\usepackage{microtype}
\theoremstyle{plain}\newtheorem{theorem}{Theorem}[section]
\newtheorem{prop}[theorem]{Proposition}

\newtheorem{lemma}[theorem]{Lemma}
\newtheorem{problem}[theorem]{Problem}

\renewenvironment{proof}[1][Proof]{\textbf{#1.} }{\ \rule{0.5em}{0.5em} \par }
\theoremstyle{remark}

\theoremstyle{definition}
\newtheorem{remark}[theorem]{Remark}

\def\EE{\mathbb{E}}

\def\be{{\beta}}

\def\si{{\sigma}}

\def\al{{\alpha}}

\def\be{{\beta}}

\def\ga{{\gamma}}

\def\si{{\sigma}}

\def\th{{\theta}}
\def\th{{\theta}}
\def\si{{\sigma}}
\def\al{{\alpha}}

\renewcommand{\theequation}{\arabic{section}.\arabic{equation}}

\renewcommand{\theequation}{\arabic{section}.\arabic{equation}}
\let\Section=\section
\def\section{\setcounter{equation}{0}\Section}

\begin{document}
\subjclass[2000]{     } 
\keywords{Fractional Brownian motion, Fractional Ornstein-Uhlenbeck, Parameter estimation, Malliavin calculus, Ergodicity, Stationary processes, Newton method, Central limit theorem.
  }
\thanks{Supported by   NSERC discovery grant}

\title[Estimation of   all parameters in fractional  
OU model]{Estimation Of all parameters in the Fractional  
Ornstein-Uhlenbeck model  under 
 discrete observations
}
\author{El Mehdi Haress }
\address{
     University of Paris-Saclay, Gif-sur-Yvette, France}
\email{el-mehdi.haress@student.ecp.fr
 }  
\author{  Yaozhong Hu}
\address{Department of Mathematical and  Statistical
 Sciences, University of Alberta at Edmonton,
Edmonton, Canada, T6G 2G1}
  \email{ yaozhong@ualberta.ca
}

  
\date{} 

\maketitle
\begin{abstract}
   Let the  Ornstein-Uhlenbeck process $(X_t)_{t\ge0}$ driven by a fractional Brownian motion $B^{H }$, described    by   $dX_t = -\theta  X_t dt + \sigma dB_t^{H }$ be observed at discrete time instants
   $t_k=kh$, $k=0, 1, 2, \cdots, 2n+2 $.  We propose   ergodic type 
statistical estimators $\hat \theta_n $, $\hat H_n $ and $\hat \sigma_n $ to estimate all  the parameters $\theta $, $H $ and $\sigma $ in the above 
Ornstein-Uhlenbeck  model  simultaneously.  We prove the strong consistence  
and the rate of convergence of   the estimators.   The step size $h$
can be arbitrarily  fixed and will not be forced to go zero, which is usually a reality.   The tools to use 
are the generalized moment  approach (via ergodic theorem) and the Malliavin calculus. 
\end{abstract}

\section{Introduction}
The Ornstein-Uhlenbeck process $(X_t)_{t\ge0}$ is described by the following Langevin equation:
\begin{equation}
    dX_t = -\theta  X_t dt+ \sigma  \text{d} B_t^{H }\,,\label{e.1.1}
\end{equation}
where $\theta>0$ so that the process is ergodic  and where  for simplicity of the presentation we   assume $X_0 = 0$. 
Other initial value can be treated   exactly in  the same way.   
We assume that  the process $(X_t)_{t\ge0}$ is observed at discrete time instants $t_k=kh $ and we want to use the observations
$\{X_{h}, X_{2h}, \cdots, X_{2n+2)h}\}$ to estimate the parameters
$\theta$, $H$ and $\sigma$ appeared in the above Langevin equation simultaneously. 

Before we continue let us briefly recall some recent relevant works  obtained in literature. 
Most of the works are concerned with the estimator 
of the drift parameter $\theta$.  
When   the Ornstein-Uhlenbeck process $(X_t)_{t\ge0}$ can be 
    observed continuously  and when the parameters $\sigma$ and $H$ are assumed to be known,   we have the following works.    
\begin{itemize}
    \item[1.]  The maximum likelihood estimator for $\theta $ defined by $\theta_T^{{\rm mle}}$
    is studied  \cite{tudor2007statistical}
    (see also the references therein for earlier references),  and is proved  to be strongly consistent. The asymptotic behavior of the bias and the mean square of $\theta_T^{{\rm mle}}$ is also given. In this paper, a strongly consistent estimator of $\sigma $ is also proposed. 
    \item[2.] A least squares estimator defined by $\Tilde{\theta}_T = \frac{-\int_0^T  X_t dX_t}{\int_0^T X_t^2}$ was studied
    in   \cite{chenhuwang,hu2010parameter,hunualartzhou}. It is proved that $\Tilde{\theta}_T \rightarrow \theta $ almost surely as $T \rightarrow \infty$.
    It is also proved  that when $H<\le 3/4$, $\sqrt{T}(\Tilde{\theta}_T-\theta )$ converges in law to a mean zero normal random variable. The variance of this normal variable is also obtained. When $H\ge 3/4$, the rate of convergence is also known
    \cite{hunualartzhou}. 
\end{itemize}

Usually in reality the process can only be  observed at discrete times 
$\{t_k=kh, k= 1, 2, \cdots, n\}$ for some fixed observation 
time lag $h>0$.  In this very interesting case,   there are
  very limited works. Let us only mention two
(\cite{hu2013parameter,panloup2019general}).  
To the best of knowledge there is only one work  \cite{brouste2013parameter} to estimate
all the parameters $\theta$, $H$ and $\si$ in the same time, where the observations are assumed to be made 
continuously.  

The  diffusion coefficient   $\si$   represents the 
``volatility"  and it is commonly believed that it should be
computed (hence estimated) by the $1/H$ variations
(see \cite{hunualartzhou} and references therein).  To use the $1/H$ variations   one has to 
assume the process can be observed continuously
(or we have high frequency data).  Namely, it is a common belief that $\sigma$ can only be estimated when one has high frequency data. 

In this work, we assume that the process can only be  observed at discrete times 
$\{t_k=kh, k= 1, 2, \cdots, n\}$ for some fixed observation 
time lag $h>0$ (without the requirement that  $h\rightarrow 0$). We want to estimate
$\theta$, $H$ and $\si$ simultaneously. The idea we use is the ergodic theorem, namely,
we find the explicit form of the 
limit  distribution  of $\frac{1}{n} \sum_{k=1}^n f(X_{kh}) $ and use 
it to estimate our parameters.  People may naturally think that 
if we appropriately choose three different $f$, then we may obtain three
different equations to obtain all the three parameters 
$\theta$, $H$ and $\si$.

However, this is impossible since   as long as we proceed this way, we shall find out that whatever we choose $f$, we cannot get independent 
equations. Motivated by a recent work \cite{chenghulong}, we may try to
 add 
the limit distribution of $\frac{1}{n}\sum_{k=1}^n g(X_{kh}, X_{(k+1)h})  $  to find 
all the parameters. However, this is still impossible because 
regardless how we choose $f$ and $g$ we obtain only two 
independent equations.  This is because regardless how we choose $f$ and $g$ the limits depends only on the covariance of the limiting Gaussians (see $Y_0$ and $Y_h$ ulteriorly). Finally, we propose to use the 
following quantities to estimate  all the three parameters 
$\theta$, $H$ and $\si$:
%
\begin{equation}\label{1.2}
    \frac{\sum_{k=1}^n X_{kh}^2}{n},\quad  \frac{\sum_{k=1}^n X_{kh}X_{(kh+h}}{n}, \quad \frac{\sum_{k=1}^n X_{2kh} X_{2kh+ 2h}}{n}.
\end{equation}
We shall study the strong consistence and joint limiting law of our estimators. 

 The paper is organized as follows. In Section \ref{sec:2}, we
 recall some known results. The construction and the strong consistency of the estimators are provided in Section \ref{sec:3}.  Central limit theorems are obtained in Section \ref{sec:4}. 
 To make the paper more readable, we delay some proofs 
 in Append \ref{A}. To use our estimators we need 
 the determinant of some functions to be nondegenerate. 
 This is given in Appendix 
 \ref{B}.  Some numerical 
 simulations to validate our estimators are illustrated
  in Appendix \ref{C}.

\section{Preliminaries}\label{sec:2}

%

Let    $(\Omega,\mathcal{F},\mathbb{P})$ be a complete probability space. The expectation on this space is denoted by $\EE $. The fractional Brownian motion   $(B_t^H, t\in \mathbb{R})$ with Hurst parameter $H\in (0,1)$  is a zero mean Gaussian process with the
following covariance structure:
\begin{equation}
    \EE (B_t^H B_s^H) = R_H(t,s) = \frac{1}{2}(\mid t \mid^{2H} + \mid s \mid^{2H} + \mid t-s \mid^{2H} ).
\end{equation}
On stochastic analysis of this fractional Brownian motion,
such as stochastic integral
$\int_a^b f(t) dB_t^H$, chaos expansion,  and stochastic differential equation
$dX_t=b(X_t)dt+\si(X_t) dB_t^H$ we refer to  \cite{bhoz}. 

For any $s, t\ge 0$, we define 
\begin{equation}
\langle I_{[0, t]}, I_{[0, s]}\rangle _{\mathcal{H}}
= R_H(s,t)\,.
\end{equation}
We can   first extend this scalar product 
to general elementary functions
$f(\cdot)=\sum_{i=1}^n a_i I_{[0, s_i]}(\cdot)$ by (bi-)linearity  
and then to general function by a limiting argument. 
We can then obtain  the reproducing kernel Hilbert space,
denoted by $ \mathcal{H}$,  
associated with this Gaussian process $B_t^H$ 
(see e.g. \cite{hu2010parameter} for more details).  
%
%




Let $\mathcal{S}$ be the space of smooth and cylindrical random variables of the form
$$F = f(B^H(\phi_1),....,B^H(\phi_n)), \quad \phi_1,...,\phi_n \in C^\infty_0([0,T]),  $$
where $f\in C_b^{\infty}(\mathbf{R}^n)$
and $B^H(\phi)=\int_0^\infty \phi(t) dB_t^H$. For such a variable $F$, we define its Malliavin derivative as the $\mathcal{H}$ valued random element:

$$DF = \sum_{k=1}^n \frac{\partial f}{\partial x_i}(B^H(\phi_1),...,B^H(\phi_n))\phi_i . $$

%


%

We shall use the following result in Section \ref{sec:4} to obtain the central
limit theorem. We refer to \cite{hubook} and many other references 
for a proof. 
\begin{prop}\label{2.2}  Let $\{F_n, n \ge 1\}$ be a sequence of random variables in the space of $p$th Wiener Chaos, $p\ge 2$ ,such that $\lim_{n\rightarrow \infty} \EE (F_n^2) = \sigma^2$. Then the following statements are equivalent: 
\begin{itemize}
    \item[(i)] $F_n$ converges in law to $N(0,\sigma^2)$ as $n$ tends to infinity.
    \item[(ii)] $\|DF_n\|_{\mathcal{H}}^2 $ converges in $L^2$ to a constant as n tends to infinity.  
\end{itemize}
\end{prop}


\section{Estimators of $\theta $,$H $ and $\sigma $}\label{sec:3}
If $X_0 = 0$, then the solution $X_t$ 
to \eqref{e.1.1} can be expressed as  
\begin{equation}
    X_t = \sigma  \int_0^t e^{-\theta (t-s)} \text{d}B^{H }_s.
    \label{e.3.1}
\end{equation}
The associated stationary solution, the solution of \eqref{e.1.1} with the the initial 
value 
\begin{equation}
Y_0=\int_{-\infty}^0 e^{ \theta s}\text{d}B_s \,, 
\end{equation} 
can be expressed as 
\begin{equation}
Y_t = \int_{-\infty}^t e^{-\theta(t-s)}\text{d}B_s
=e^{-\theta t} Y_0+X_t\label{e.3.3} 
\end{equation}
 and has the same distribution as 
  the limiting normal distribution of $X_t$
  (when $t\rightarrow \infty$). 
Let's consider the following two  quantities :
\begin{equation}
\left\{
    \begin{array}{lll}
    \eta_n =  \frac{1}{n}\sum_{k=1}^n X_{kh}^2, 
      \\
    \eta_{h, n} =  \frac{1}{n}\sum_{k=1}^n X_{kh}X_{ kh+ h}\,,
   \\
 \eta_{2h, n} =  \frac{1}{n}\sum_{k=1}^n X_{2kh}X_{ 2kh+ 2h}   \,.  
  \end{array} \label{e.eta}
\right.
\end{equation}

%
As in \cite{cheridito2003fractional} or \cite{hu2013parameter}, we have 
the following ergodic result: 
\begin{equation}
\lim_{n\rightarrow\infty} \eta_n = \EE (Y_0^2) = \sigma ^2 \frac{\Gamma(2H +1) \sin   (\pi H )}{2\pi} \int_{-\infty}^{\infty} \frac{\mid x\mid^{1-2H }}{\theta ^2 + x^2}dx\,.
\label{e.3.6}
\end{equation} 
Now we 
want to have a similar result for $\eta_{h,n}$.  
First, let's study the ergodicity of the processes $\{Y_{t+h}-Y_t\}_{t \ge 0}$.   
 According to \cite{magdziarz2011ergodic},    a centered Gaussian wide-sense stationary process $M_t$ is ergodic if $\EE (M_t M_0) \rightarrow 0$ as $t$ tends to infinity.  
We shall apply this result to  $M_t=Y_{t+h}-Y_t\,, t \ge 0$.
Obviously, it  is a centered Gaussian stationary process and 
\[
\EE ((Y_{t+h}-Y_t)(Y_{h}-Y_0)) = \EE (Y_{t+h}Y_{h}) - \EE (Y_{t+h}Y_0) - \EE (Y_t Y_h) + \EE (Y_tY_0)\,.
\]
In  \cite[Theorem 2.3]{cheridito2003fractional},  it is proved  that $\EE (Y_tY_0) \rightarrow 0$ as $t$ goes to infinity. Thus, it is easy to see that $\EE ((Y_{t+h}-Y_t)(Y_{h}-Y_0)) \rightarrow 0$.   
Hence, we see that the process $\{Y_{t+h}-Y_t\}_{t \ge 0}$ is ergodic. 
This implies 
\[
\frac{\sum_{k=1}^n \large{[}Y_{(k+1)h} - Y_{kh}\large{]^2}}{n} \rightarrow_{n\rightarrow\infty} \EE ([Y_h - Y_0]^2)\,. 
\]

This combined with \eqref{e.3.6} yields  the following Lemma.  
\begin{theorem}\label{t.3.1} Let $\eta_n$,   $\eta_{h,n}$   and 
$\eta_{2h,n}$
be defined by \eqref{e.eta}. Then
  as $n\rightarrow \infty$  we have 
\begin{eqnarray}
\lim_{n\rightarrow\infty} \eta_n &=& \EE (Y_0^2) = \sigma ^2 \frac{\Gamma(2H +1) \sin   (\pi H )}{2\pi} \int_{-\infty}^{\infty} \frac{\mid x\mid^{1-2H }}{\theta ^2 + x^2}dx\,;\label{e.3.a6}\\ 
 \lim_{n\rightarrow\infty} \eta_{h,n}
 &=& \EE (Y_0Y_h) = \sigma ^2 \frac{\Gamma(2H +1) \sin   (\pi H )}{2\pi} \int_{-\infty}^{\infty} e^{ixh} \frac{\mid x\mid^{1-2H }}{\theta ^2 + x^2}dx\,; 
\label{e.3.a7}\\
\lim_{n\rightarrow\infty} \eta_{2h,n}
 &=& \EE (Y_0Y_{2h} ) = \sigma ^2 \frac{\Gamma(2H +1) \sin   (\pi H )}{2\pi} \int_{-\infty}^{\infty} e^{2ixh} \frac{\mid x\mid^{1-2H }}{\theta ^2 + x^2}dx\,. 
 \nonumber\\ 
 \label{e.3.a8}
\end{eqnarray}  
\end{theorem} 
From the above theorem we propose the following  construction for the
estimators of the parameters $\th$, $H$ and $\si$.

First let us define 
\begin{equation}
\left\{
    \begin{array}{lll}
        f_1(\sigma,\theta,H) := \frac{1}{2\pi}\sigma^2\Gamma(2H+1)\sin   (\pi H)\int_0^{\infty} \frac{x^{1-2H}}{\theta^2+x^2}dx\,;\\
        f_2(\sigma,\theta,H) := \frac{1}{2\pi}\sigma^2\Gamma(2H+1)\sin   (\pi H)\int_0^{\infty} \cos(hx)\frac{x^{1-2H}}{\theta^2+x^2}dx\,;  \\
        f_3(\sigma,\theta,H) :=  \frac{1}{2\pi}\sigma^2\Gamma(2H+1)\sin   (\pi H\int_0^{\infty} \cos(2hx)\frac{x^{1-2H}}{\theta^2+x^2}dx  
    \end{array}
\right.  \label{e.3.9} 
\end{equation}
and let   $f(\theta,H,\sigma) =(f_1(\sigma,\theta,H), f_2(\sigma,\theta,H), f_3(\sigma,\theta,H))^T$. 

Then we set 
\begin{equation}
\left\{
    \begin{array}{lll}
        f_1(\sigma,\theta,H)  =  \eta_n=   \frac{1}{n}\sum_{k=1}^n X_{kh}^2 \,;\\
        f_2(\sigma,\theta,H)  = \eta_{h, n} =\frac{1}{n}\sum_{k=1}^n X_{kh}X_{ kh+ h}  \,;  \\
        f_3(\sigma,\theta,H) =  \eta_{2h, n} =\frac{1}{n}\sum_{k=1}^n X_{2kh}X_{ 2kh+ 2h}  \,. 
    \end{array}
\right.   \label{e.3.10} 
\end{equation}
This is a system of three equations for three unknowns
$(\th, H, \si)$. If  the determinant of the Jacobian (of $f$) 
\begin{equation}
J(\th, H, \si)= \left(\begin{matrix}
\frac{\partial f_1(\th, H, \si)}{\partial \th}
&\frac{\partial f_1(\th, H, \si)}{\partial H}&\frac{\partial f_1(\th, H, \si)}{\partial \si}\\
 \frac{\partial f_2(\th, H, \si)}{\partial \th}
&\frac{\partial f_2(\th, H, \si)}{\partial H}&\frac{\partial f_2(\th, H, \si)}{\partial \si}\\
\frac{\partial f_3(\th, H, \si)}{\partial \th}
&\frac{\partial f_3(\th, H, \si)}{\partial H}&\frac{\partial f_3(\th, H, \si)}{\partial \si}\\
\end{matrix}\right)\  \label{e.3.jacobian} 
\end{equation} 
is not zero at  $(\th_0, H_0, \si_0)$,  then 
the above equation \eqref{e.3.10} has a unique solution
$(\tilde \th_n , \tilde H_n , \tilde \si_n )$
in the neighborhood of $(\th_0, H_0, \si_0)$
(We will discuss  the determinant 
$J(\th, H, \si)$ in Appendix \ref{B}).  Namely,  
$(\tilde \th_n , \tilde H_n , \tilde \si_n )$ satisfies
\begin{equation*}
\left\{
    \begin{array}{lll}
        f_1(\tilde \th_n , \tilde H_n , \tilde \si_n)  =  \eta_n=   \frac{1}{n}\sum_{k=1}^n X_{kh}^2 \,;\\
        f_2(\tilde \th_n , \tilde H_n , \tilde \si_n)  = \eta_{h, n} =\frac{1}{n}\sum_{k=1}^n X_{kh}X_{ kh+ h}  \,;  \\
        f_3(\tilde \th_n , \tilde H_n , \tilde \si_n) =  \eta_{2h, n} =\frac{1}{n}\sum_{k=1}^n X_{2kh}X_{ 2kh+ 2h}  \,. 
    \end{array}
\right.    
\end{equation*}
Or we can write as
\begin{equation}
(\tilde \th_n , \tilde H_n , \tilde \si_n)^T=f^{-1}(\Upupsilon_n)\,,
\end{equation}
where $f^{-1}$  is the inverse function of $f$ (if it exists) and 
\begin{equation}
\Upupsilon_n=(\eta_n, \eta_{h, n} ,  \eta_{2h, n})^T\,. 
\end{equation}
We shall use $(\tilde \th_n , \tilde H_n , \tilde \si_n )$  to estimate  
the parameters 
$(  \th  ,   H  ,   \si  )$.   
 We call $(\tilde \th_n , \tilde H_n , \tilde \si_n )$  the ergodic 
 (or generalized moment) estimator of   
$(  \th  ,   H  ,   \si  )$.  
 
It seems hard to explicitly obtain the explicit solution 
of the system of  equation \eqref{e.3.10}.  However, it is a classical  algebraic 
equations. 
There are numerous 
numeric approaches to find the approximate solution.
We shall give some validation of our estimators
numerically in Appendix \ref{C}.

%
%
%


Since $f$ is a continuous function of $(\theta , H , \sigma) $
the inverse function $f^{-1}$ is also continuous if it exists. 
Thus we have the following  a strong consistency result which is an immediate consequence of Theorem \ref{t.3.1}. 
\begin{theorem}  Assume \eqref{e.3.10} has a unique solution
$(\tilde \th_n , \tilde H_n , \tilde \si_n )$.  Then $(\tilde \th_n , \tilde H_n , \tilde \si_n )$ converge almost surely to $(\theta , H , \sigma) $ respectively as $n$ tends to infinity.
\end{theorem}


\section{Central limit theorem}\label{sec:4}
In this section, we shall concern with the central limit theorem associated with our ergodic estimator 
 $(\tilde \th_n , \tilde H_n , \tilde \si_n )$.
We shall prove that  $\sqrt{n} (\tilde \th_n-\th , \tilde H_n -H, \tilde \si_n-\si )$  converge   in law to a mean zero normal vector.  

%

Let's  first consider the random variable $F_n$ defined by 
\begin{equation}
    F_n = \begin{pmatrix}
 \sqrt{n}(\eta_n - \EE (\eta_n)) \\
\sqrt{n}(\eta_{h,n} - \EE (\eta_{h,n})) \\
\sqrt{n}(\eta_{2h,n} - \EE (\eta_{2h,n}))  \\
\end{pmatrix}.
\end{equation}

Our first goal is to show that $F_n$ converges in law to a multivariate normal distribution using Proposition \ref{2.2}. So we consider a linear combination: 
\begin{equation}
    G_n = \alpha \sqrt{n}(\eta_n - \EE (\eta_n)) + \beta \sqrt{n}(\eta_{h,n} - \EE (\eta_{h,n})) + \gamma \sqrt{n}(\eta_{2h,n} - \EE (\eta_{2h,n})),
    \label{e.4.4} 
\end{equation}
and show that it converges to a normal distribution.  

We will use the following Feynman diagram formula \cite{hubook}, where 
interested readers can find a proof.  
\begin{prop}\label{p.feynman_diagram_0} Let $X_1,   X_2, X_3, X_4$ be   jointly Gaussian random variables with mean zero. Then  
\[
\EE (X_1 X_2 X_3 X_4) = \EE (X_1X_2)\EE (X_3X_4) +  \EE (X_1X_3)\EE (X_2X_4)+ \EE (X_1X_4)\EE (X_2X_3)\,.
\]
\end{prop} 
An immediate  consequence of this result  is
\begin{prop}\label{p.feynman_diagram} Let $X_1,   X_2, X_3, X_4$ be   jointly Gaussian random variables with mean zero. Then   
\begin{empheq}[left=\empheqlbrace]{align}   
 &\EE \left[ (X_1 X_2 -\EE(X_1X_2))(X_3 X_4-\EE (X_3 X_4))\right]\nonumber\\
 &\qquad\qquad  =\EE(X_1X_3)\EE(X_2X_4)+\EE(X_1X_4)\EE(X_2X_3)\,; \\
 &\EE \left[ (X_1 ^2 -\EE(X_1^2))(X_2 X_3-\EE (X_2 X_3))\right]   =2\EE(X_1X_2)\EE(X_1X_3) \,; \\
& \EE \left[ (X_1^2-\EE(X_1^2))(X_2^2  -\EE (X_2 ^2))\right]   =2\left[ \EE(X_1X_2)\right]^2\,.  
\end{empheq} 
\end{prop}

\begin{theorem} \label{theorem4.3}
Let $H \in (0,\frac{1}{2})\cup(\frac{1}{2},\frac{3}{4})$. Let $X_t$ be the Ornstein-Uhlenbeck process defined by equation \eqref{e.1.1}   and let $\eta_n$, $\eta_{h,n}$, $\eta_{2h,n}$ 
be   defined by \eqref{e.eta}. Then
\begin{equation}
\begin{pmatrix}
 \sqrt{n}(\eta_n - \EE (\eta_n)) \\
\sqrt{n}(\eta_{h,n} - \EE (\eta_{h,n})) \\
\sqrt{n}(\eta_{2h,n} - \EE (\eta_{2h,n}))  \\
\end{pmatrix} \rightarrow N(0,\Sigma ),
\end{equation}  
where $\Sigma=\left(\Sigma(i,j)\right)_{1\le i,j\le 3}$ is a symmetric matrix whose elements are given by 
\begin{empheq}[left=\empheqlbrace]{align}   
 &\Sigma(1,1)=\Sigma(2,2)=\Sigma(3,3)=2\left[\EE(Y_0^2)\right]^2+4\sum_{m=0}^\infty  \left[\EE(Y_0Y_{mh})\right]^2\,; 
\label{e.4.9}\\
 & \Sigma(1,2)=\Sigma(2,1)=\Sigma(2,3) =\Sigma(3,2)=
 4\sum_{m=0}^\infty   \EE(Y_0Y_{mh}) \EE(Y_0Y_{(m+1)h}) \,;\label{e.4.10} \\
 & \Sigma(1,3)=    \Sigma(3,1)=
 4\sum_{m=0}^\infty   \EE(Y_0Y_{2mh}) \EE(Y_0Y_ {2mh+2h}) \,. 
 \label{e.4.11}
\end{empheq} 
\end{theorem} 
\begin{remark} 
\begin{enumerate}
\item[(1)] It is easy   from  the following proof to see that all entries $\Sigma(i,j)$ of the covariance matrix $\Sigma$ 
are   finite. 
\item[(2)] In an earlier work of Hu and Song  it is said 
\cite[Equation (19.19)]{hu2013parameter}   that the variance $\Sigma$ 
(corresponding to our $\Sigma(1,1)$ in our notation) is independent of the time lag
$h$. But there was an error on the bound of $A_n$ on \cite[page 434,
line 14]{hu2013parameter}. So, $A_n$ there 
  does not go to zero. Its limit is re-calculated  in this work. 
  \item[(3)] It is clear that we can use
  $\tilde \eta_{2h, n} =\frac{1}{n}\sum_{k=1}^n X_{kh}X_{ kh+ 2h} $ to replace 
  $\eta_{2h,n}$. 
  \item[(4)] We may use \eqref{e.3.a6}-\eqref{e.3.a8} to compute $\Sigma(i,j)\,, 1\le i , j\le 3$ explicitly.  
\end{enumerate} 
\end{remark}
\begin{proof}
We write 
\[
\EE (G_n^2) = (\al, \be, \ga)  \Sigma_n(\al, \be, \ga)^T\,,
\quad 
\Sigma_n=\left(\Sigma_n(i,j)\right)_{1\le i, j\le 3}\,,
\]
where $\Sigma_n$ is  a symmetric $3\times 3$ matrix given by 
\begin{empheq}[left=\empheqlbrace]{align}   
\Sigma_n(1,1)
&=  n\EE\left[(\eta_n-\EE(\eta_n ))^2\right]\,;  \nonumber \\
\Sigma_n(1,2)
&=  \Sigma_n(2,1)=n\EE\left[ (\eta_n-\EE(\eta_n ))(\eta_{h,n}-\EE(\eta_{h,n} ))\right]\,;\nonumber \\
\Sigma_n(1,3)
&=  \Sigma_n(3,1)=n\EE\left[ (\eta_n-\EE(\eta_n ))
(\eta_{2h,n}-\EE(\eta_{2h,n} ))\right]\,;\nonumber \\
\Sigma_n(2,2)
&=  n\EE\left[ (\eta_{h,n}-\EE(\eta_{h,n} )) ^2\right]\,;\nonumber  \\
\Sigma_n(2,3)
&=  \Sigma_n(3,2)= n\EE\left[ (\eta_{h,n}-\EE(\eta_{h,n} )) 
(\eta_{2h,n}-\EE(\eta_{2h,n} )) \right]\,;\nonumber \\ 
\Sigma_n(3,3)
&=   n\EE\left[  
(\eta_{2h,n}-\EE(\eta_{2h,n} ))^2 \right]  \,. \nonumber 
\end{empheq}
It is easy to observe that 
\begin{enumerate}
\item the limits of $\Sigma_n(1,1)$,  
$\Sigma_n(2,2)$, and $\Sigma_n(3,3)$  are the same;
\item the limits of $\Sigma_n(1,2)$,
 and $\Sigma_n(2,3)$  are the same;  
\item the limit  of $\Sigma_n(1,3)$ can be obtained from the limit of 
$\Sigma_n(1,2)$ by replacing $h$ by $2h$;  
\item the matrix is symmetric. 
\end{enumerate}
Thus, we only need to compute the limits  of $\Sigma_n(1,1)$  and $\Sigma_n(1,2)$.  

First, we compute the limit of $\Sigma_n(1,1)$.
From the definition \eqref{e.eta} of $\eta_n$ and Proposition \ref{p.feynman_diagram}, we have
\begin{eqnarray}
\Sigma_n(1,1)
&=&\frac{1}{n} \sum_{k,k'=1}^n \EE \left[(X_{kh}^2-\EE
\left[(X_{kh})^2\right] )(X_{k'h}^2-\EE
\left[(X_{k'h})^2\right] )\right] \nonumber  \\
&=&\frac{2}{n} \sum_{k,k'=1}^n \left[\EE (X_{kh}  X_{k'h} )\right]^2
\end{eqnarray}
By Lemma \ref{l.a.1}, we see that 
\[
\Sigma_n(1,1)\rightarrow \Sigma (1,1)=2\left[\EE(Y_0^2)\right]^2+4\sum_{m=0}^\infty  \left[\EE(Y_0Y_{mh} )\right]^2\,.
\]
This proves \eqref{e.4.9}

Now let consider the limit of $\Sigma_n(1,2)$. From the 
definitions \eqref{e.eta}   and   from   Proposition \ref{p.feynman_diagram}  it follows
\begin{eqnarray}
\Sigma_n(1,2)
&=&\frac{1}{n} \sum_{k,k'=1}^n \EE \left[(X_{kh}^2-\EE
\left[\left(X_{kh})^2\right] \right)\left(X_{k'h} X_{(k'+1)h}-\EE
\left[X_{k'h} X_{(k'+1)h}\right] \right)\right] \nonumber  \\
&=&\frac{2}{n} \sum_{k,k'=1}^n  \EE (X_{kh}  X_{k'h} ) \EE (X_{kh}  X_{(k'+1)h} ) \,.
\end{eqnarray}
By Lemma \ref{l.a.2}, we have
\begin{eqnarray}
\Sigma_n(1,2)\rightarrow 4\sum_{m=0}^\infty   \EE(Y_0Y_{mh}) \EE(Y_0Y_{(m+1)h})\,. 
\end{eqnarray}
This proves \eqref{e.4.10}.  \eqref{e.4.11} is obtained from 
\eqref{e.4.10} by replacing $h$ by $2h$.  This proves 
\begin{equation}
\lim_{n\rightarrow \infty} 
\EE (G_n^2) = (\al, \be, \ga)  \Sigma(\al, \be, \ga)^T\,. 
\end{equation}
Using Lemma \ref{l.a.4}, we know that $J_n :=\langle DG_n,DG_n\rangle _\mathcal{H}$ converges to a constant.  Then by Proposition 
\ref{2.2}, we know $G_n$ converges in law to a normal random variable. 

Since $G_n$ converges to a normal for any $\al$, $\be$, and $\ga$, we know by the Cram\'er-Wold theorem that 
$F_n$ converges to a mean zero Gaussian random vector, 
proving the theorem. 
\end{proof}
Now using the   delta method  and the above  Theorem 
\ref{theorem4.3} we immediately have the following theorem. 
\begin{theorem}\label{theorem4.4}
Let $H \in (0,\frac{1}{2})\cup(\frac{1}{2},\frac{3}{4})$. Let $X_t$ be the Ornstein-Uhlenbeck process defined by equation \eqref{e.1.1}   and let $(\tilde \th_n , \tilde H_n , \tilde \si_n )$ be defined by \eqref{e.3.10}. Then
$$\begin{pmatrix}
 \sqrt{n}(\tilde{\theta_n} - \theta) \\
\sqrt{n}(\tilde H_n - H) \\
\sqrt{n}(\tilde \sigma_n  - \si)  \\
\end{pmatrix} { \stackrel{d} \rightarrow}  N(0, \tilde \Sigma)\\
 \,,
$$
where $J$ denotes the Jacobian matrix of $f$,  defined by 
\eqref{e.3.jacobian},  $\Sigma$ is defined in \ref{theorem4.3}  and
\begin{eqnarray}
\tilde \Sigma=\left[J(\th, H, \si)\right]^{-1}  \Sigma  \left[J^T(\th, H, \si)\right]^{-1}  \,. 
\end{eqnarray} 
\end{theorem}

%
%
%
%
%

%
%
%
%
%
%
%
%
%

\bibliographystyle{abbrv}

\bibliography{bibliography}

\bigskip
\appendix
\renewcommand{\theequation}{A.\arabic{equation}} 
\def\section{\setcounter{equation}{0}\Section}
\section{Detailed computations }\label{A}

First, 
we need a lemma
from \cite[supplementary data, Lemma 5.4, Equation (5.7)]{hu2010parameter}.   
\begin{lemma}\label{lemma.a.1} Let $X_t$ be the Ornstein-Uhlenbeck process 
defined by \eqref{e.1.1}.  Then 
\begin{equation}
|\EE(X_tX_s)|\le C(1\wedge |t-s|^{2H-2})
\le   (1+ |t-s| )^{2H-2} \,. \label{e.a.9}
\end{equation}
The above inequality also holds true for $Y_t$. 
\end{lemma}

\begin{lemma}\label{l.a.1}  Let $X_t$ be defined by \eqref{e.1.1}.  When $H \in (0,\frac{1}{2})\cup (\frac{1}{2},\frac{3}{4})$ we have
\begin{equation}
\lim_{n\rightarrow \infty} \frac{1}{n} \sum_{k,k'=1}^n \left[\EE (X_{kh}X_{k'h})\right]^2 =\left[\EE(Y_0^2)\right]^2 
+2\sum_{m=1}^\infty \left[ \EE (Y_{0}Y_{mh })\right] ^2\,. 
\end{equation}
 \end{lemma}
\begin{proof} To simplify notations we shall use $X_k$, $Y_k$ to represent 
$X_{kh}$, $Y_{kh}$ etc. 
From the relation \eqref{e.3.3} it is easy to see that 
%
%
\begin{eqnarray}
\EE (X_{k}X_{k'}) 
&=& \EE (Y_{k}Y_{k'})   
 - e^{-\theta k'h} \EE (Y_0 Y_k )  -e^{-\theta kh} \EE (Y_0 Y_{k'} )+  e^{-\theta (k+k')h} \EE ( Y_0^2  )\nonumber\\ 
 &=:&\sum_{i=1}^4 I_{i, k,k'} \,,\label{e.5.1} 
\end{eqnarray} 
where $ I_{i, k, k'}$, $i=1, \cdots, 4$,  denote the above $i$-th term.

Let us first consider $\frac{1}{n} \sum_{k,k'=1}^n I_{i, k,k'}^2 $  
for $i=2, 3, 4$.  First, we consider $i=2$. 
 By \cite[Theorem 2.3]{cheridito2003fractional},
we  know that $\EE (Y_0Y_{k})$ converges to $0$ when $k\rightarrow \infty$.
Thus by the Toeplitz theorem, we have 
\begin{eqnarray}
\frac1n\sum_{k,k'=1}^n I_{2, k, k'}^2
=  \frac1n \sum_{k,k'=1}^n e^{-2\theta k' h}\left[ \EE (Y_0 Y_k )\right]^2
\le C\frac1n \sum_{k }^\infty \left[ \EE (Y_0 Y_k)\right]^2    \rightarrow 0\,. \label{e.5.2} 
\end{eqnarray}
Exactly in the same way   we have 
\begin{eqnarray}
\frac1n\sum_{k,k'=1}^n I_{3, k, k'}^2    \rightarrow 0\,. \label{e.5.3} 
\end{eqnarray}
When $i=4$,   we have easily 
\begin{eqnarray}
\lim_{n\rightarrow \infty} 
\frac1n\sum_{k,k'=1}^n I_{4, k, k'}^2 = 
\frac1n\sum_{k,k'=1}^n e^{-2\th (k+k')} 
\left[\EE(Y_0^2)\right]^2 \rightarrow 0\,.   
\label{e.5.4} 
\end{eqnarray} 
Now we   have 
\begin{eqnarray*}
\frac{1}{n} \sum_{k,k'=1}^n \left[\EE (X_{k }X_{k' })\right]^2
&=&\frac{1}{n} \sum_{i,j=1}^4\sum_{k,k'=1}^n I_{i,k,k'}I_{j, k, k'}\\
&=&  \frac{1}{n} \sum_{k,k'=1}^n I_{1,k,k'}^2+\frac{1}{n}  \sum_{i\not=1,{\rm or} j\not=1}\sum_{k,k'=1}^n I_{i,k,k'}I_{j, k, k'}\,. 
\end{eqnarray*} 
When one of the $i$ or $j$ is not equal to $1$, we have by the H\"older inequality 
\begin{eqnarray*}
\frac{1}{n}  \sum_{k,k'=1}^n \left|I_{i,k,k'}I_{j, k, k'}\right|
&\le& \left(\frac{1}{n}  \sum_{k,k'=1}^n I_{i,k,k'}^2\right)^{1/2}
\left(\frac{1}{n}  \sum_{k,k'=1}^n  I_{j, k, k'}^2\right)^{1/2}
\end{eqnarray*} 
which will go to $0$ if we can show $\frac{1}{n}  \sum_{k,k'=1}^n I_{1,k,k'}^2, n=1, 2, \cdots$ is bounded.    In fact, we have   
\begin{eqnarray}
&&\frac{1}{n}  \sum_{k,k'=1}^n I_{1,k,k'}=
\frac1n\sum_{k,k'=1}^n \left[ \EE (Y_k Y_{k'})\right]^2 \nonumber \\
&&\qquad\quad= 
\frac1n\sum_{k,k'=1}^n \left[ \EE (Y_0Y_{|k'-k|})\right]^2\nonumber \\
&&\qquad \quad =\EE(Y_0^2)+
\frac2n\sum_{m=1}^{n-1} (n-m) \left[ \EE (Y_0Y_m)\right]^2 \nonumber \\
&&\qquad\quad=  \left[\EE(Y_0^2)\right]^2+ 2\sum_{m=1}^{n-1}   \left[ \EE (Y_0Y_m)\right]^2-\frac2n\sum_{m=1}^{n-1} m  \left[ \EE (Y_0Y_m)\right]^2\,. 
\end{eqnarray}  
By Lemma \ref{lemma.a.1} for $Y_t$ or an  expression of $\EE (Y_0Y_m) $ given in \cite[Theorem 2.3]{cheridito2003fractional}:
\[
\EE (Y_0Y_m) = \frac{1}{2} \sigma^2 \sum_{n=1}^{N} (\Pi_{k=0}^{2n-1} (2H-k)) m^{2H-2n} + O(m^{2H-2N-2})\,.
\]
This means  $\EE (Y_0Y_m) = O(m^{2H-2})$ as $m\rightarrow \infty$, which in turn means that $\left[\EE (Y_0Y_m)\right] ^2 = O(m^{4H-4})$.
 Hence, for $H<\frac{3}{4}$, $\sum_{m=0}^{n-1} \EE (Y_{0}Y_{m })^2$ converges as $n$ tends to infinity.  

Notice that for $H <\frac{3}{4}$, $m \EE (Y_0Y_m)^2 =O(m^{4H-3})\rightarrow  0$  as $m\rightarrow \infty$. By Toeplitz theorem we have
\[
\frac{1}{n} \sum_{m=0}^{n-1} m  \left[ \EE (Y_0Y_m)\right]^2 \rightarrow 0\quad \hbox{as $n\rightarrow \infty$}\,.  
\]
Thus, $\frac{1}{n}\sum_{k,k'>k}^{n}\left[\EE (Y_{k}Y_{k'})\right]^2$ converges
to $\left[\EE(Y_0^2)\right]^2+2\sum_{m=1}^\infty   \left[ \EE (Y_0Y_m)\right]^2$ as $n$ tends to infinity. 
\end{proof} 
 
\begin{lemma}\label{l.a.2}  Let $X_t$ be defined by \eqref{e.1.1}.  When $H \in (0,\frac{1}{2})\cup (\frac{1}{2},\frac{3}{4})$ we have
\begin{equation}
\lim_{n\rightarrow \infty} \frac{1}{n} \sum_{k,k'=1}^n  \EE (X_{kh}X_{k'h}) 
\EE (X_{kh}X_{(k'+1)h})  =2  \sum_{m=0}^{\infty }  \EE (Y_0Y_{mh}) \EE (Y_0Y_{(m+1)h} ) \,. 
\end{equation}
 \end{lemma}
\begin{proof}
We continue to use  the notations  in Lemma \ref{l.a.1}. 
\begin{eqnarray}
\EE (X_{k}X_{k'}) 
 &= &\sum_{i=1}^4 I_{i, k,k'}\,,  \nonumber\\
 \EE (X_{k}X_{k'+1})  
 &= &\sum_{i=1}^4 I_{i, k,k'+1} \,,\label{e.5.8} 
\end{eqnarray} 
where $ I_{i, k, k'}$, $i=1, \cdots, 4$,  is defined in \eqref{e.5.2}.
As in the proof of Lemma \ref{l.a.1}, we have 
\begin{eqnarray}
&&\lim_{n\rightarrow \infty} 
\sum_{k,k'=1}^n  \EE (X_{k }X_{k' }) 
\EE (X_{k }X_{ k'+1 })=\lim_{n\rightarrow \infty}  
\frac1n\sum_{k,k'=1}^n   \EE (Y_k Y_{k' })\EE (Y_k Y_{k'+1})  \nonumber \\
&&\qquad\quad= 
\frac1n\sum_{k,k'=1}^n   \EE (Y_0Y_{|k'-k|}) \EE (Y_0Y_{|k'+1-k| }) \nonumber \\
&&\qquad\quad =
\frac1n\sum_{m=0}^{n-1} (n-m)   \EE (Y_0Y_m) \EE (Y_0Y_{m+1} )
+\frac1n\sum_{m=1}^{n-1} (n-m)   \EE (Y_0Y_m) \EE (Y_0Y_{m-1} ) \nonumber
\end{eqnarray}   
Now we can use the same argument as in proof of Lemma \ref{l.a.1}
to obtain 
\[
\lim_{n\rightarrow \infty} 
\sum_{k,k'=1}^n  \EE (X_{k }X_{k' }) 
\EE (X_{k }X_{ k'+1 })=2  \sum_{m=0}^{\infty }  \EE (Y_0Y_m) \EE (Y_0Y_{m+1} ) \,, 
\]
proving the lemma.  
\end{proof}

Let $G_n$ be   defined by \eqref{e.4.4}   in Section \ref{sec:4}.   Its Malliavin derivative is given by 
\begin{eqnarray}
 DG_n &=& \frac{1}{\sqrt{n}} 2\alpha \sum_{k=1}^n  X_{k}DX_{k} + \frac{1}{\sqrt{n}} \beta \sum_{k=1}^n ( X_{k}  DX_{k+1} 
 +  X_{k+1}DX_{k} )\nonumber\\
 &&\qquad + \frac{1}{\sqrt{n}} \sum_{k=1}^n \gamma ( X_{k } DX_{ k +2 }+  X_{ k +2 }DX_{k } )\,. \label{e.a.10}
\end{eqnarray} 
\begin{lemma}\label{l.a.4} 
Define  the sequence of random variables $J_n :=\langle DG_n,DG_n\rangle _\mathcal{H}$.  Then 
\begin{equation}
\lim_{n\rightarrow\infty} \EE\left[J_n-\EE(J_n)\right]^2= 0\,. 
 \end{equation}  
\end{lemma}
\begin{proof}
It is easy to see that $J_n 
$ is a linear combination of terms of the following forms (with the coefficients being a quadratic forms of $\al, \be, \ga$): 
\begin{eqnarray} 
&&
\tilde J_n:=\frac{1}{n} \sum_{k',k=1}^{n} \langle DX_{k_1 },DX_{k_1' }\rangle _\mathcal{H} X_{k_2  }X_{k_2'  }\nonumber\\
&&\qquad  =\frac{1}{n} \sum_{k',k=1}^{n} \EE (X_{k_1 }X_{k_1'  }) X_{k_2 }X_{k_2' }\,, \label{e.a.11}
\end{eqnarray} 
where $k_1,    k_2 $ may  take $k, k+1, k+2$, and 
$ k_1',  k_2'$ may  take $k', k'+1, k'+2$.  For example, one term is
to take $k_1= k_2=k $ and $k_1'=k'+1$, $k_2'=k' $  which corresponds 
to the product: 
\begin{eqnarray}
&&\langle \frac{1}{\sqrt{n}} 2\alpha \sum_{k=1}^n  X_{k}DX_{k}, \frac{1}{\sqrt{n}} \beta \sum_{k=1}^n ( X_{k}  DX_{k+1} \rangle \nonumber\\
&&\qquad\quad =\frac{2\al\be}{n} \sum_{k',k=1}^{n}\EE(X_k X_{k'+1})X_kX_{k'}=:2\al\be \tilde J_{0, n}\,. 
\end{eqnarray} 
We will first give a detail argument to explain why  
\[
\EE \left[ \tilde J_{0, n}-\EE( \tilde J_{0, n})\right]^2\rightarrow 0 
\]
and then we outline the  procedure that similar claims hold true for any terms in \eqref{e.a.11}. Note that  $\EE( \tilde J_{0, n})$ will not converge to $0$.

From the Proposition \ref{p.feynman_diagram} it follows
\begin{eqnarray*}
\EE \left[ \tilde J_{0, n}-\EE( \tilde J_{0, n})\right]^2
&=& \frac{1}{n^2 } \sum_{k,k', j, j'=1}^{n}\EE(X_k X_{k'+1})
\EE(X_j X_{j'+1})\EE(
X_kX_{j})\EE(X_{k'}X_{j'})\nonumber\\
&&\quad  + \frac{1}{n^2 } \sum_{k,k', j, j'=1}^{n}\EE(X_k X_{k'+1})
\EE(X_j X_{j'+1})\EE(
X_kX_{j'})\EE(X_{k'}X_{j})\nonumber\\
&=:& I_{1, n}+I_{2,n}\,.
\end{eqnarray*}
Using  \eqref{e.a.9} we have 
\begin{eqnarray*} 
I_{1, n}&\le& \frac{1}{n^2 } \sum_{k,k', j, j'=1}^{n}
(1+ |k'-k|)^{2H-2} (1+ |j'-j|)^{2H-2}  
\\
&&\qquad\quad (1+ |j-k|)^{2H-2} (1+ |k'-j'|)^{2H-2}\,;  \\
I_{2, n}&\le& \frac{1}{n^2 } \sum_{k,k', j, j'=1}^{n}
(1+ |k'-k|)^{2H-2} (1+ |j'-j|)^{2H-2}  
\\
&&\qquad\quad (1+ |j'-k|)^{2H-2} (1+ |k'-j |)^{2H-2}\,. 
\end{eqnarray*} 
Now it is elementary to see that $I_{1, n}\rightarrow 0$ 
and $I_{2, n}\rightarrow 0$  when $n
\rightarrow \infty$.   

Now we deal with  the general term  
\[
\tilde J_{1, n}:=\frac{1}{n} \sum_{k',k=1}^{n} \EE (X_{k_1 }X_{k_1'  }) X_{k_2 }X_{k_2' } 
\]
 in \eqref{e.a.11},  
where $k_1,    k_2 $ may  take $k, k+1, k+2$, and 
$ k_1',  k_2'$ may  take $k', k'+1, k'+2$. 
We use  Proposition \ref{p.feynman_diagram} to obtain 
\begin{eqnarray*}
\EE \left[ \tilde J_{1, n}-\EE( \tilde J_{1, n})\right]^2
&=&  \frac{1}{n^2 } \sum_{k,k', j, j'=1}^{n}\EE (X_{k_1 }X_{k_1'  })    \EE (X_{j_1 }X_{j_1'  }) \EE( X_{k_2 }X_{j_2  }) 
\EE( X_{k_2' }X_{j_2'  }) \nonumber\\
&&\quad  +\frac{1}{n^2 } \sum_{k,k', j, j'=1}^{n}\EE (X_{k_1 }X_{k_1'  })    \EE (X_{j_1 }X_{j_1'  }) \EE( X_{k_2 }X_{j_2' }) 
\EE( X_{k_2' }X_{j_2  }) \nonumber\\ 
&=:& \tilde I_{1, n}+\tilde I_{2,n}\,, 
\end{eqnarray*} 
where $k_1,    k_2   $ may  take $k, k+1, k+2$, and 
$ k_1',  k_2'   $ may  take $k', k'+1, k'+2$,
$j_1,    j_2   $ may  take $j, j+1, j+2$, and 
$ j_1',  j_2' $ may  take $j', j'+1, j'+2$. 
 Using  \eqref{e.a.9} we have 
\begin{eqnarray*} 
\tilde I_{1, n}&\le& \frac{1}{n^2 } \sum_{k,k', j, j'=1}^{n}
(1+ |k'-k|)^{2H-2} (1+ |j'-j|)^{2H-2}  
\\
&&\qquad\quad (1+ |j-k|)^{2H-2} (1+ |k'-j'|)^{2H-2} \,; \\
\tilde I_{2, n}
&\le& \frac{1}{n^2 } \sum_{k,k', j, j'=1}^{n}
(1+ |k'-k|)^{2H-2} (1+ |j'-j|)^{2H-2}  
\\
&&\qquad\quad (1+ |j'-k|)^{2H-2} (1+ |k'-j |)^{2H-2}\,. 
\end{eqnarray*} 
Now it is elementary to see that $I_{1, n}\rightarrow 0$ 
and $I_{2, n}\rightarrow 0$  when $n
\rightarrow \infty$.    
\end{proof}

 \bigskip 
\section{Determinant of the Jacobian   of $f$}\label{B}

In this section we plot the determinant of the Jacobian of $J$. The determinant is plotted as a function of two parameters when the third one is fixed. Roughly, what these figures show is that the determinant will approach zero as $\theta$ goes to $\infty$ and as $\sigma$ goes to $0$ (See Fig \ref{fig:1} and Fig \ref{fig:2}). They also show that the determinant becomes negative when $H$ is roughly smaller that $0.3$ or too close to $1$ (See Fig \ref{fig:3} and Fig \ref{fig:4}). 

Therefore, if we suppose that $0.3<H<\frac{3}{4}$, since the true value of the parameter $\theta$ is not infinite and $\sigma>0$, we should be able to say that the determinant of $J(\theta,H,\sigma)$ is positive in the neighborhood of $(\theta_0,H_0,\sigma_0)$.

\begin{figure}[h!]
    \centering
    \includegraphics[scale=0.4]{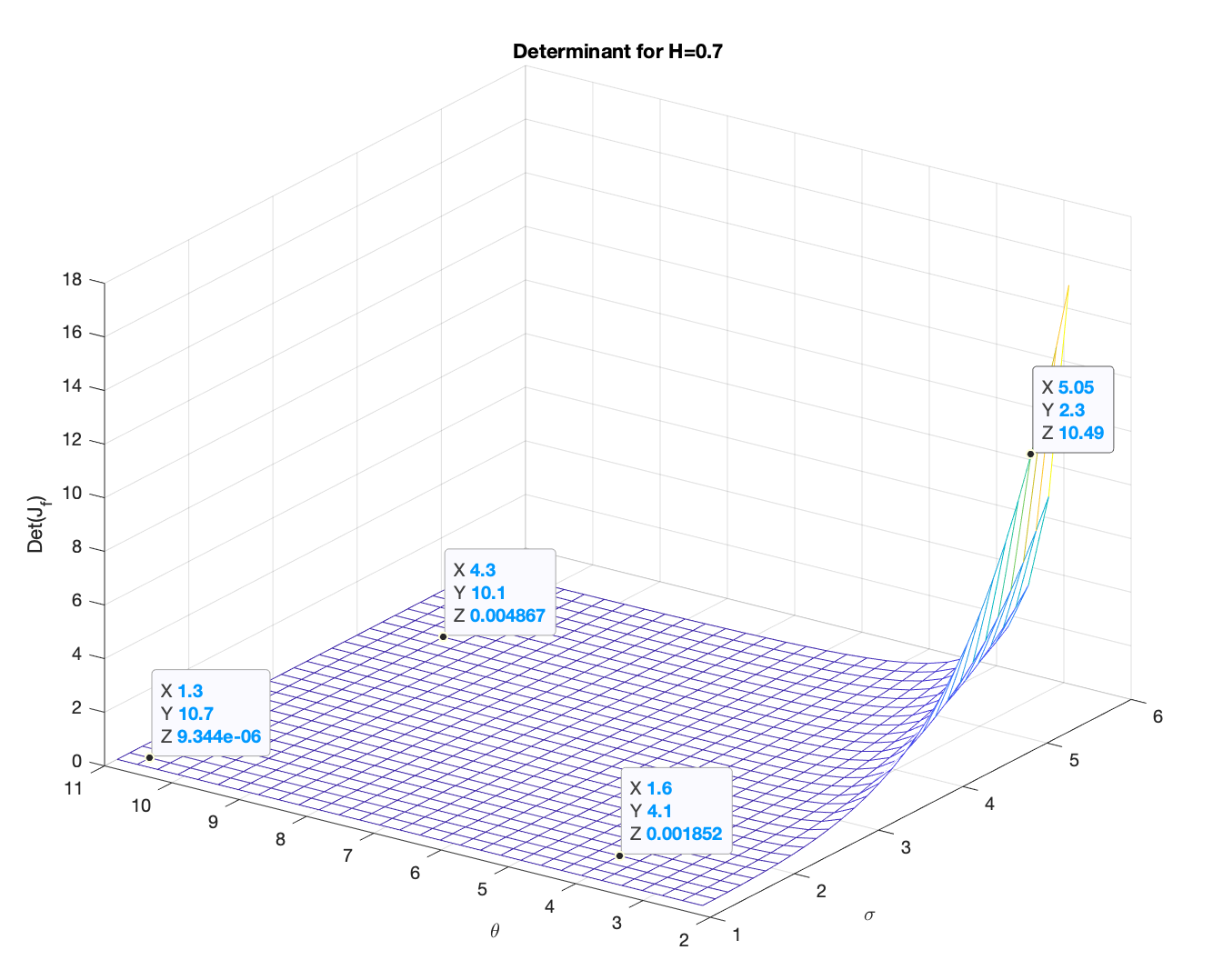}
    \caption{Determinant of the jacobian matrix of $f$ for $H =0.7$}
    \label{fig:1}
\end{figure}

\begin{figure}[h!]
    \centering
    \includegraphics[scale=0.4]{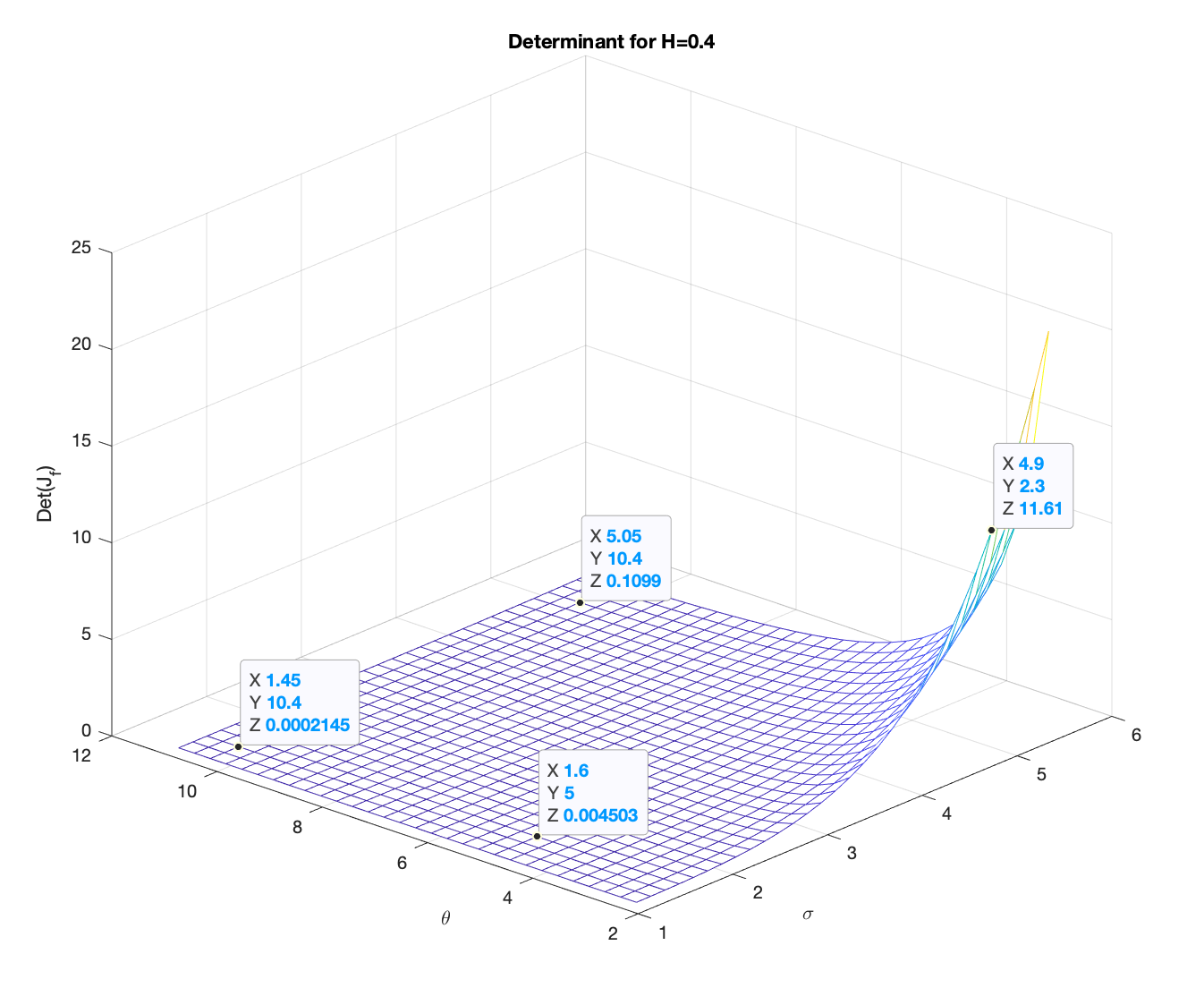}
    \caption{Determinant of the jacobian matrix of $f$ for $H = 0.4$ }
    \label{fig:2}
\end{figure}

\begin{figure}[h!]
    \centering
    \includegraphics[scale=0.5]{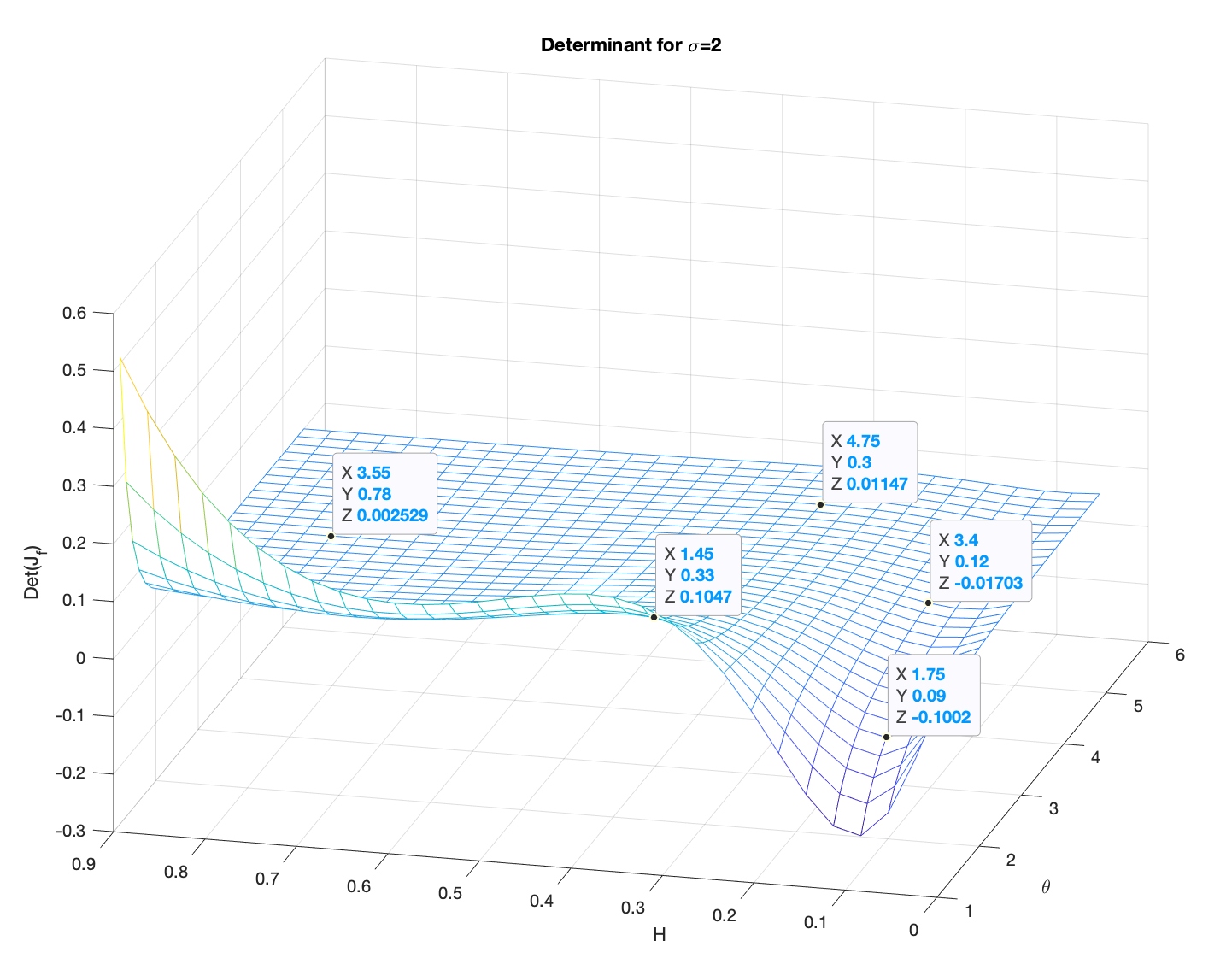}
    \caption{Determinant of the jacobian matrix of $f$ for $\sigma =2$}
    \label{fig:3}
\end{figure}

\begin{figure}[h!]
    \centering
    \includegraphics[scale=0.45]{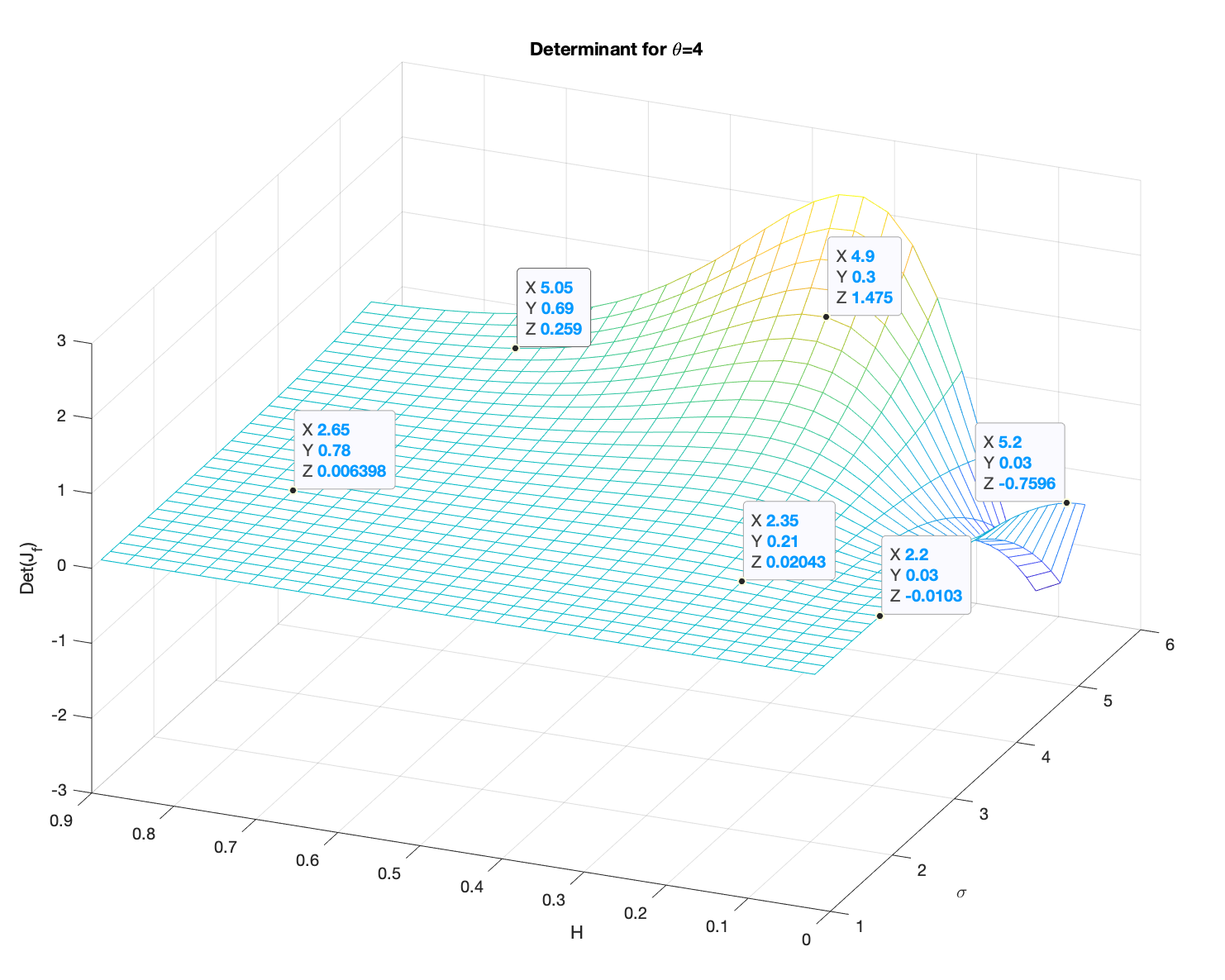}
    \caption{Determinant of the jacobian matrix of $f$ for $\theta =5$}
    \label{fig:4}
\end{figure}

\section{Numerical results}\label{C}

\subsection{Strong consistency of the estimators}
In this subsection, we illustrate the almost-sure convergence by plotting different trajectories of the estimators.  
We observe  that when $log_2(n) \ge 14$, the estimators become very close to the true parameter. 

However, since our estimators are random (they depend on the sample $\{X_{kh}\}_{k=1}^n$), what's important to see in these figures is the deviations from the true parameter we are estimating. Even if three trajectories are not enough to make statements about the variance, the figures predict that the variance of $\tilde{\theta}_n$ is very high compared to the other estimators (see Fig \ref{fig:5} and Fig \ref{fig:6})and that, for $H$ close to $0$ (see Fig \ref{fig:7}), the deviations of $\tilde{H}_n$ increase.

\begin{figure}[h!]
    \centering
    \includegraphics[scale=0.4]{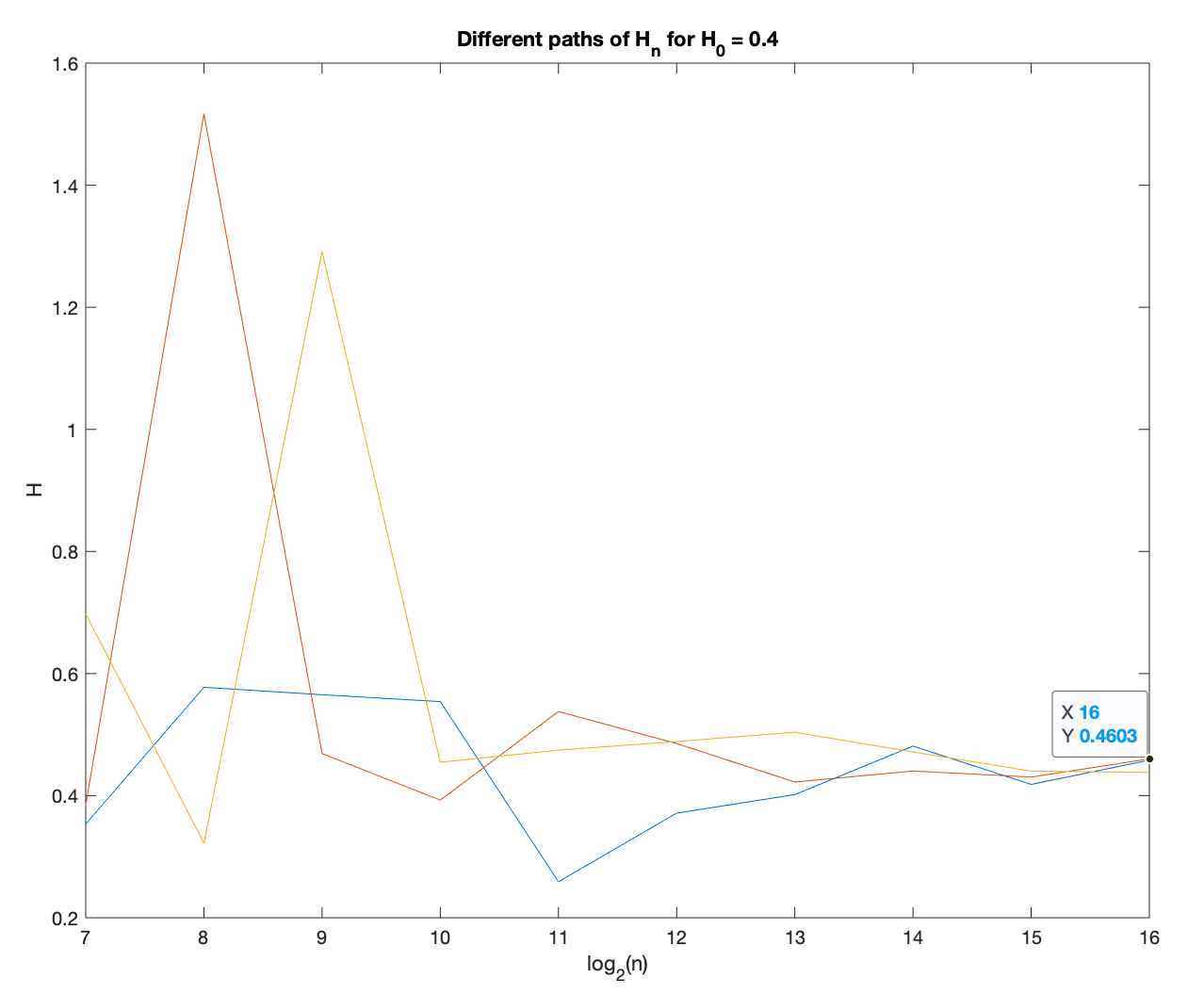}
    \includegraphics[scale=0.4]{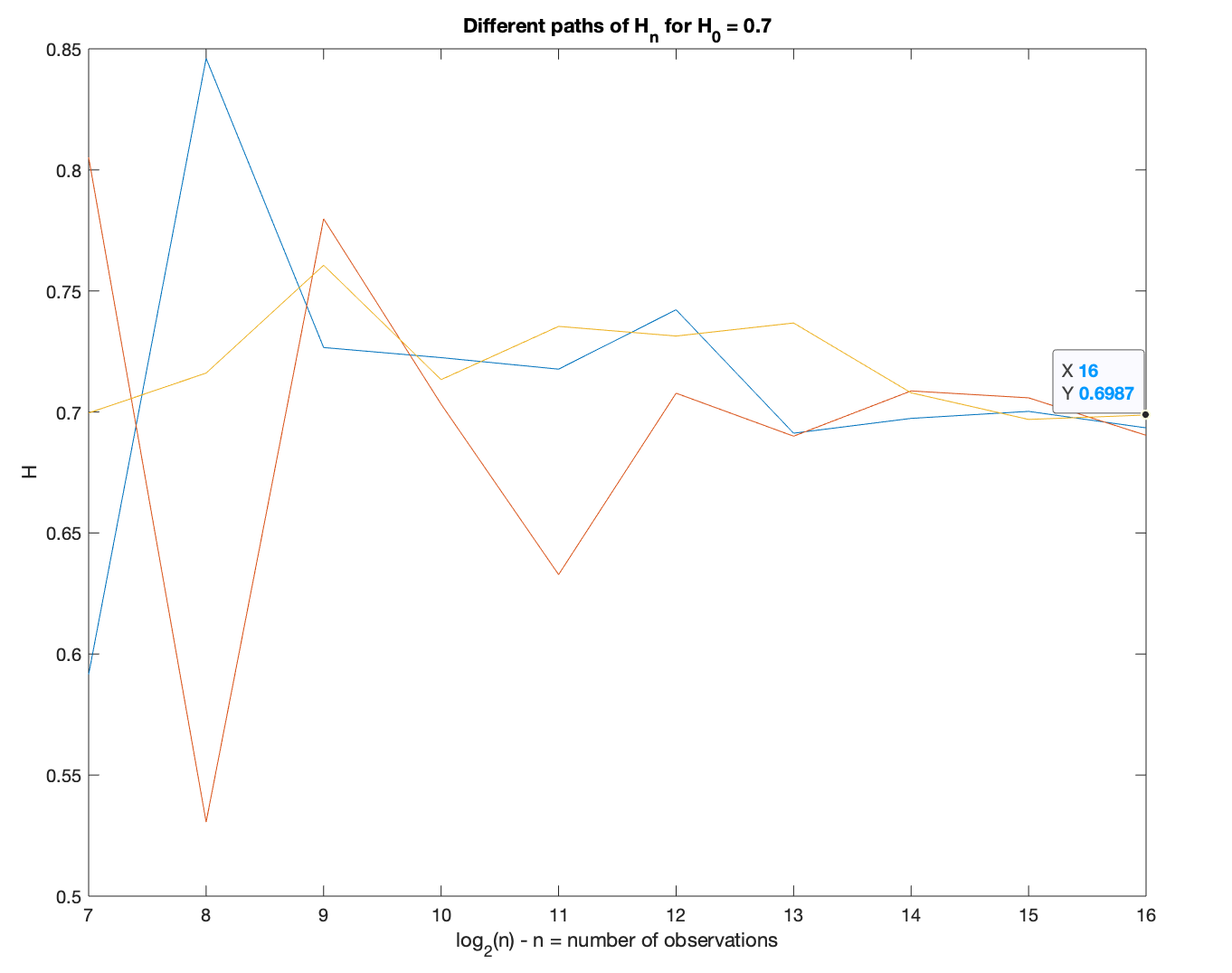}
    \caption{Convergence of $\widetilde{H_n}$ for $H =0.7$ and $H =0.4$}
    \label{fig:5}
\end{figure}

\begin{figure}[h!]
    \centering
    \includegraphics[scale=0.42]{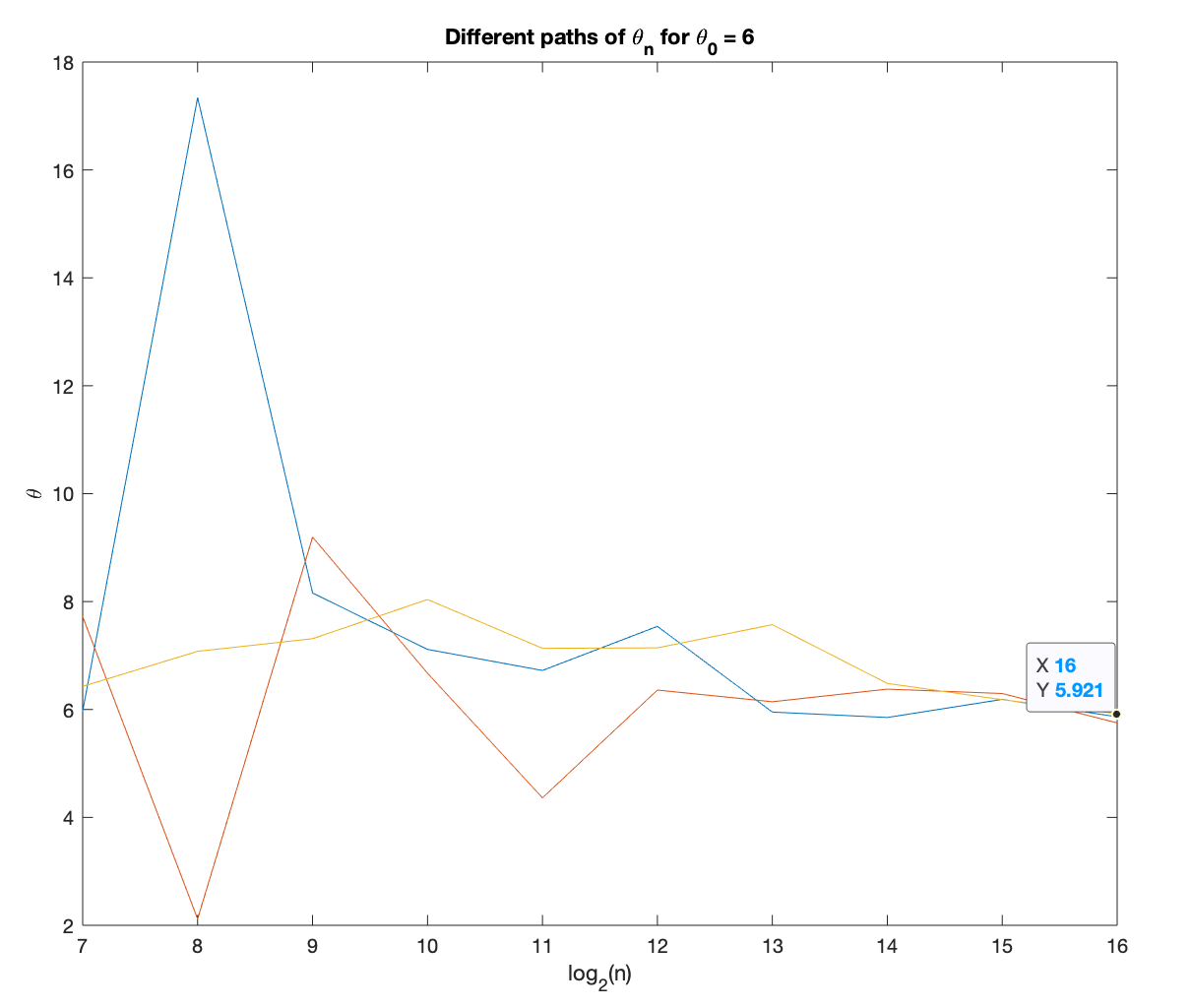}
    \caption{Convergence of $\widetilde{\theta_n}$ for $\theta =6$}
    \label{fig:6}
\end{figure}

\begin{figure}[h!]
    \centering
    \label{Fig7}
    \includegraphics[scale=0.46]{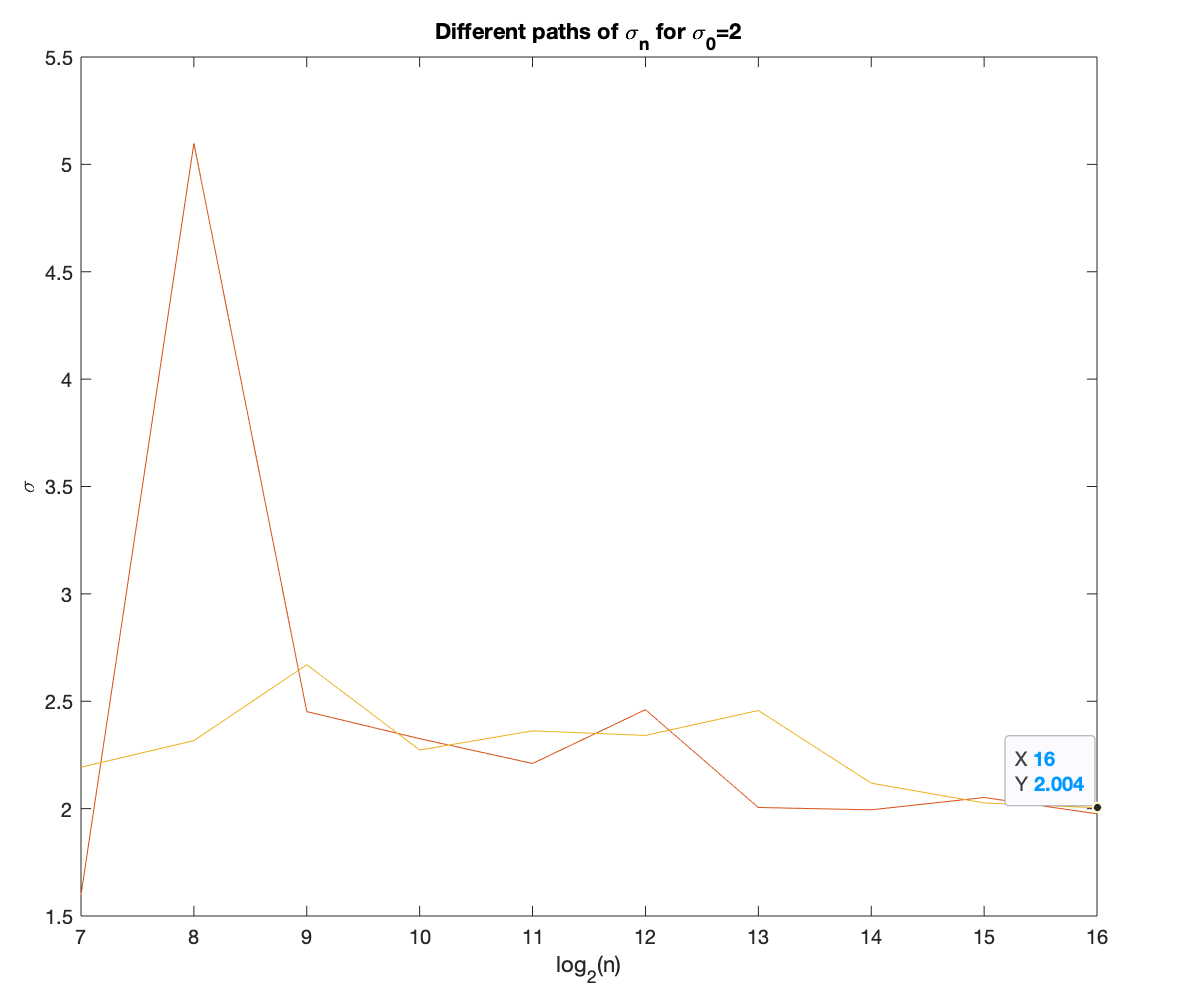}
    \caption{Convergence of $\widetilde{\sigma_n}$ for $\sigma =2$}
    \label{fig:7}
\end{figure}

\subsection{Mean and standard deviation/Asymptotic behavior of the estimators}

It is important to check the mean and deviation of our estimators. For example, a large variance implies a large deviation and therefore a ``weak" estimator. That is why we plotted the mean and variance of our estimators for $n=2^{12}$ over $100$ samples.  

As we observe, the standard deviation (s.d) of $\tilde{\theta}_n$ is larger than the s.d of $\tilde{sigma}_n$ which is larger than the s.d of $\tilde{H}_n$ (see table \ref{table1} and table \ref{table2}). Notice also that the s.d of $\tilde{H}_n$ increases as $H$ decreases.  

In \cite{hu2013parameter}, the variance of the $\theta$ estimator is proportional to $\theta^2$. In our case, it is difficult to compute the variances of our estimators (they depend on the matrix $\Sigma$ (see Theorem \ref{theorem4.3}) and the Jacobian of the function $f$ (see equation \eqref{e.3.9}), however we should probably expect something similar which could explain the gap in the variances since the values of $\theta$ are usually bigger that the values taken by $\sigma$ or $H$.

Having access to $100$ estimates of each parameter, we   are  also   able to plot the distributions of our estimators to show that they effectively have a Gaussian nature (\ref{theorem4.4}). (Fig \ref{fig:8}, Fig \ref{fig:9} and Fig \ref{fig:10}).

\begin{table}[h!]
\begin{center}
  \caption{$H =0.7$,$\theta =6$ and $\sigma=2$}
  \begin{tabular}{ l | c | r }
    
                           & Mean & Standard deviation \\ \hline
    $\widetilde{H_n}   $  &    0.704   & 0.0221 \\ \hline
    $\widetilde{\theta_n}$ &   6.2983   & 0.8288 \\ \hline
    $\widetilde{\sigma_n}$ & 2.0921  &  0.2117 \\ \hline
    \hline
  \end{tabular}
  \label{table1}

\end{center}
\end{table}

\begin{table}[h!]
\begin{center}
 \caption{$H =0.4$,$\theta =6$ and $\sigma=2$}
  \begin{tabular}{ l | c | r }
    
                           & Mean & Standard deviation \\ \hline
    $\widetilde{H_n}   $  &    0.4392   & 0.0531 \\ \hline
    $\widetilde{\theta_n}$ & 6.832  &  1.3227 \\ \hline
    $\widetilde{\sigma_n}$ &   2.4785   & 0.3833 \\ \hline
    \hline
  \end{tabular}
  \label{table2}
 
\end{center}
\end{table}

\begin{figure}[h!]

    \includegraphics[scale=0.36]{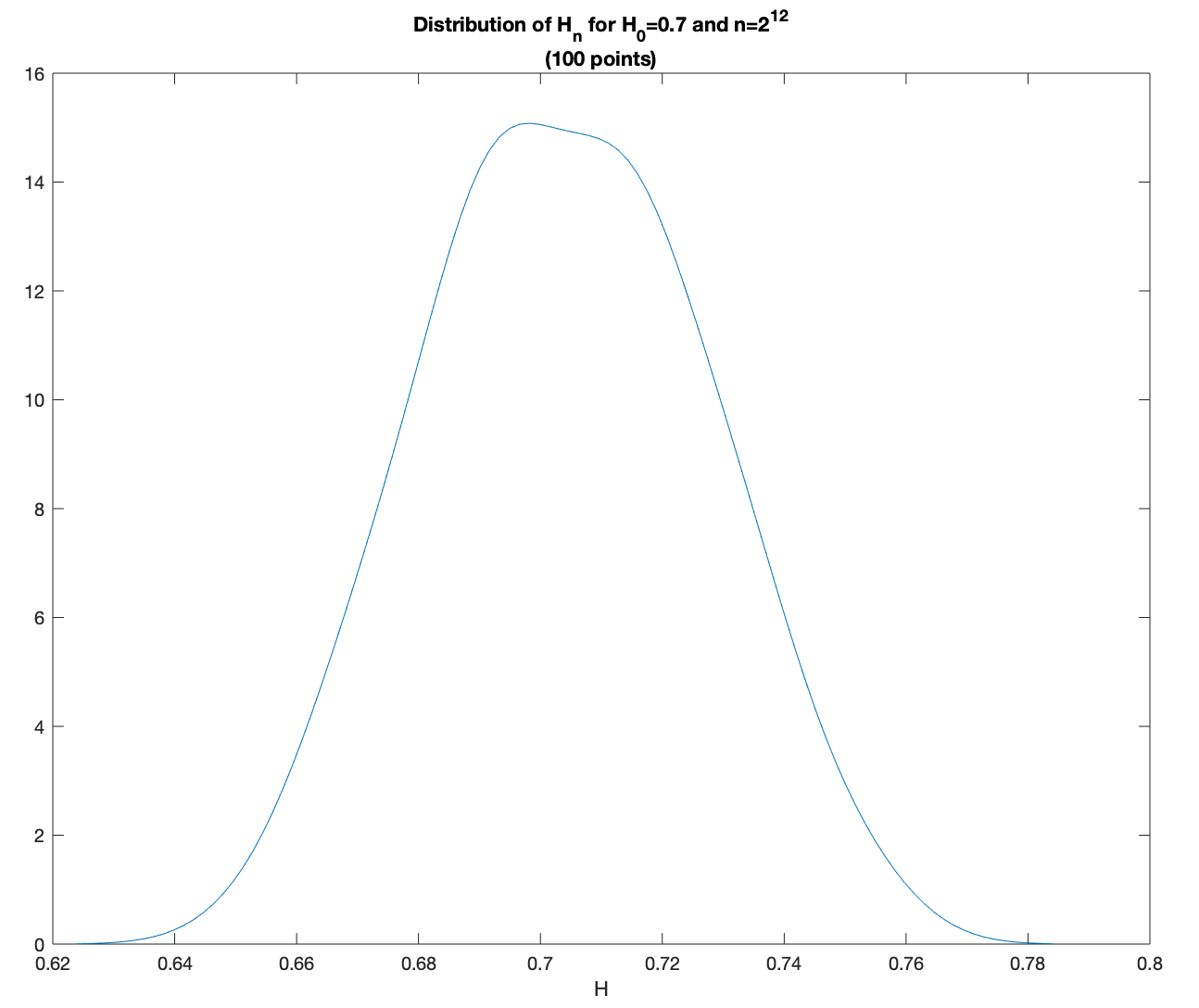}
    \includegraphics[scale=0.38]{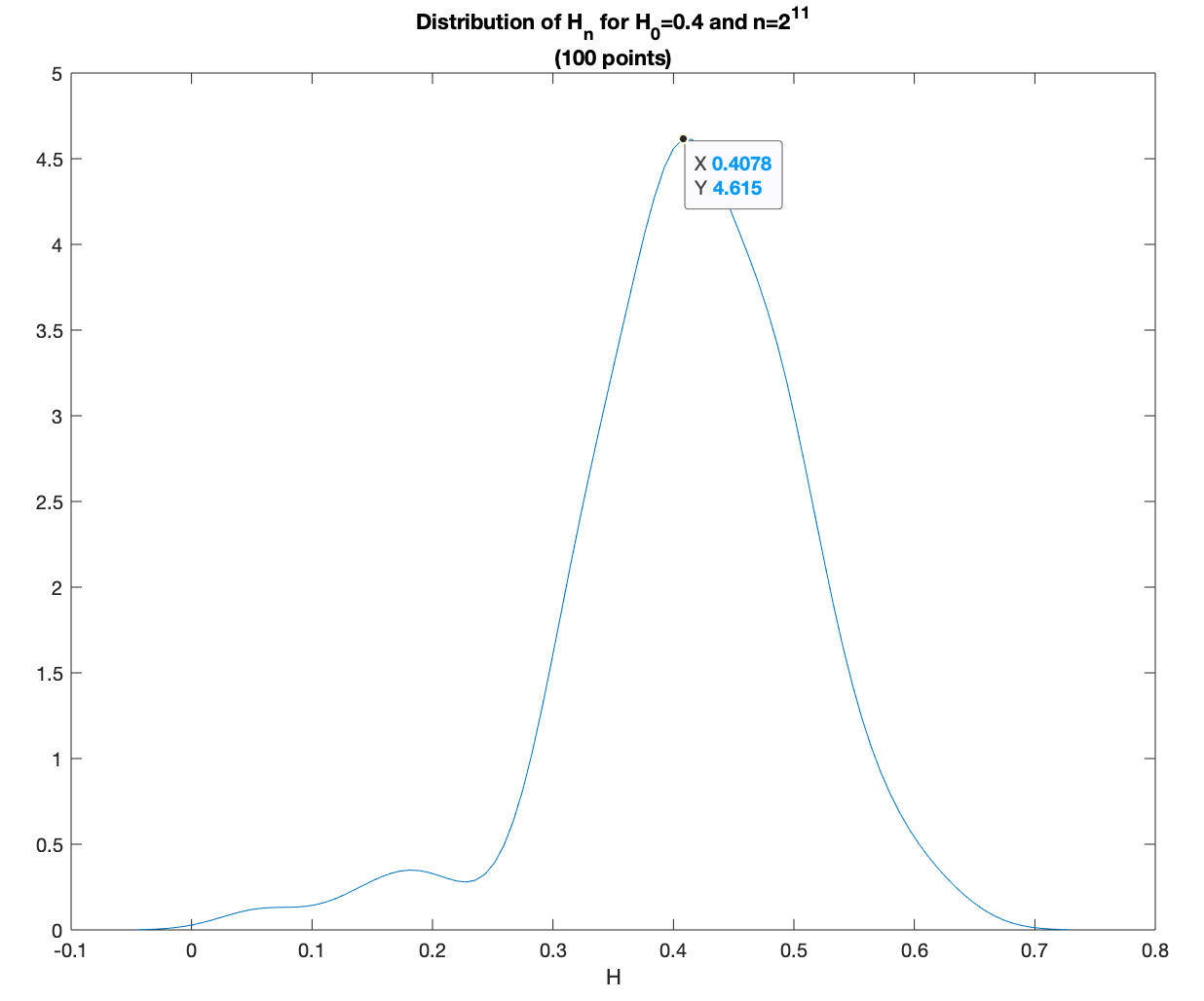}
    \caption{Distribution of $\widetilde{H_n}$ for $H =0.7$ and $H=0.4$ }
    
    \label{fig:8}
\end{figure}

\begin{figure}[h!]
    \centering
    \includegraphics[scale=0.5]{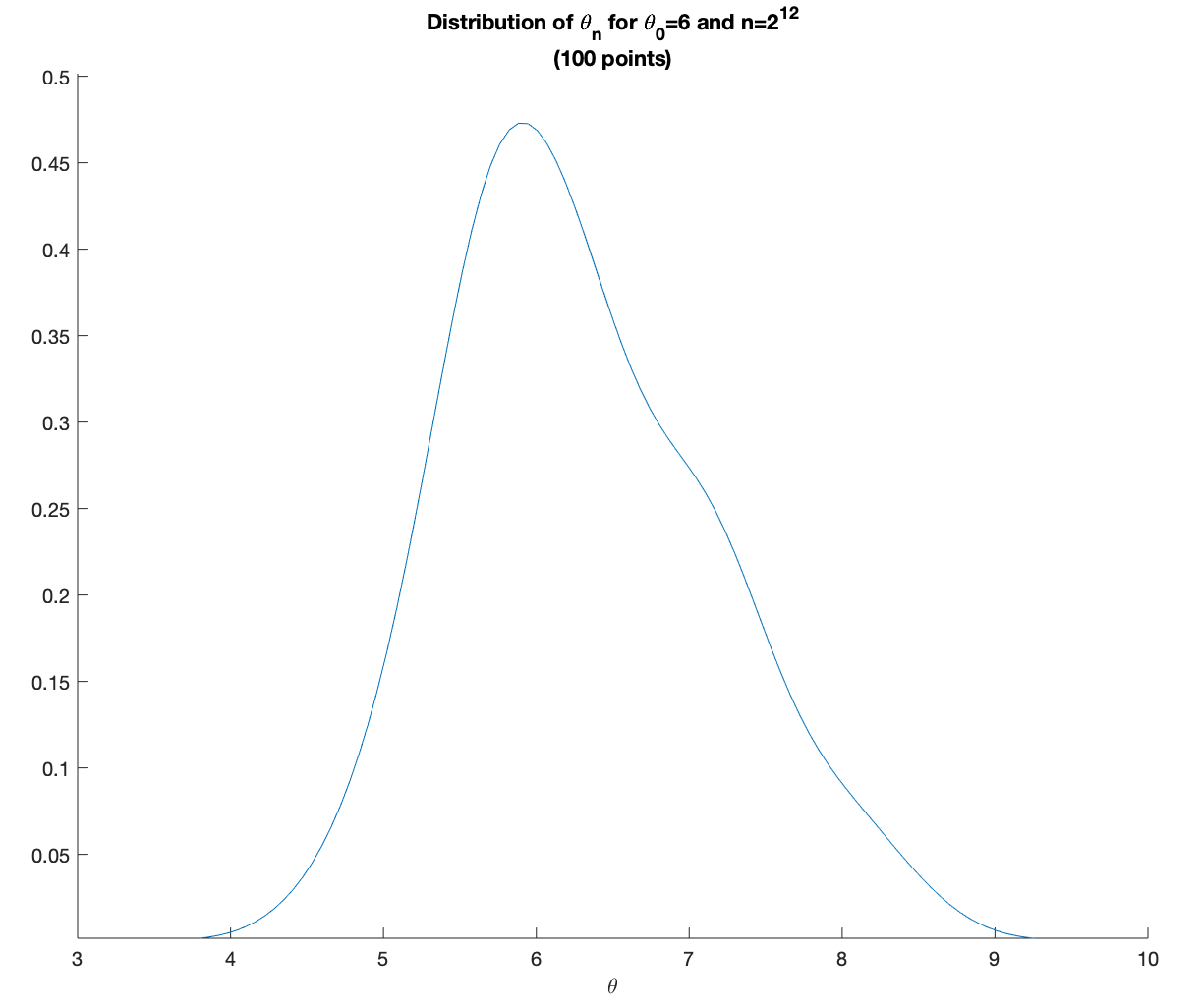}
    \caption{Distribution of $\widetilde{\theta_n}$ for $\theta =6$}
    \label{fig:9}
\end{figure}

\begin{figure}[h!]
    \centering
    \includegraphics[scale=0.5]{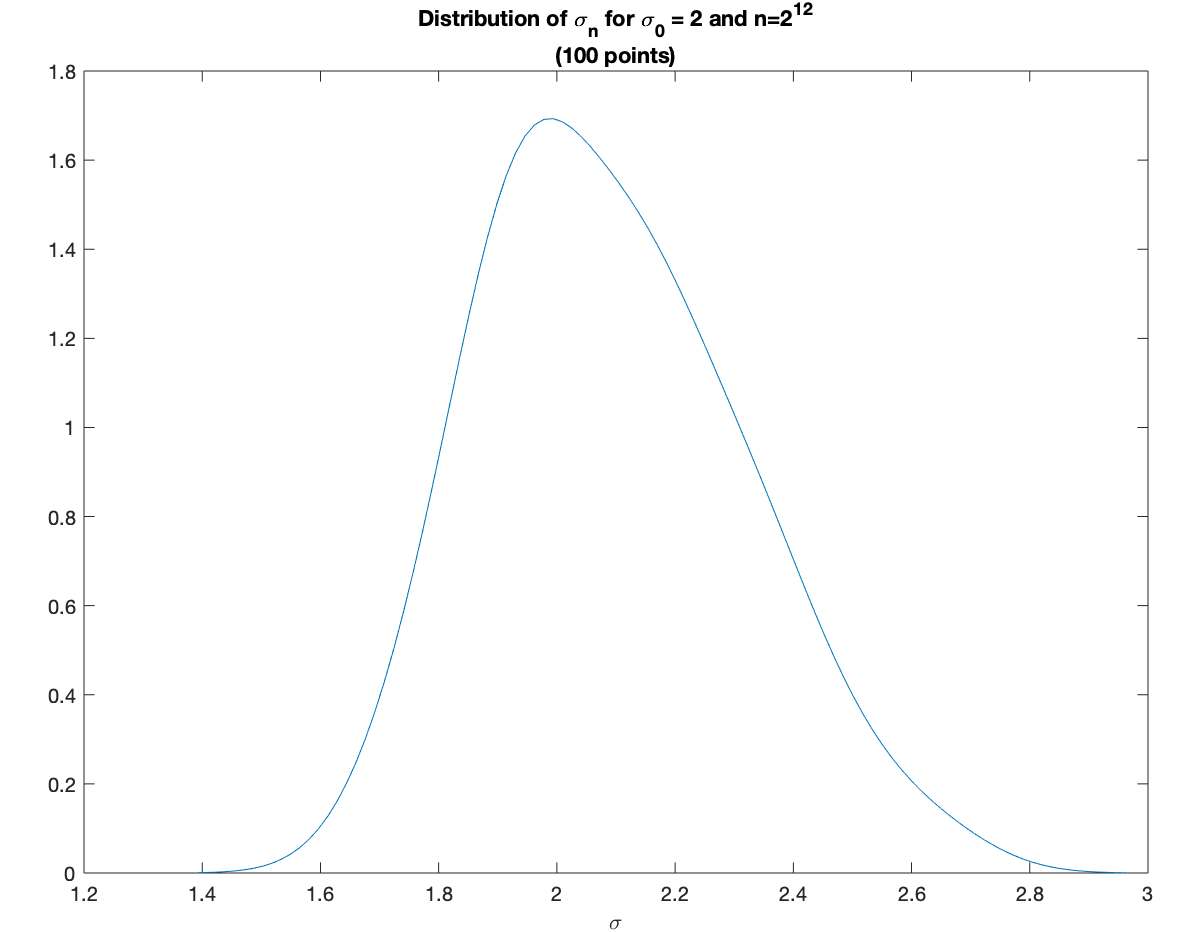}
    \caption{Distribution of $\widetilde{\sigma_n}$ for $\sigma =2$}
    \label{fig:10}
\end{figure}

\end{document}